\documentclass[12pt,a4paper]{amsart}
\usepackage{amssymb}
\usepackage{mathtools}
\usepackage[all]{xy}
\newcommand{\CC}{{\mathbb{C}}}

\newcommand{\FF}{{\mathbb{F}}}

\newcommand{\GG}{{\mathbb{G}}}

\newcommand{\NN}{{\mathbb{N}}}

\newcommand{\ZZ}{{\mathbb{Z}}}

\newcommand{\bm} {\mathbf m}

\newcommand{\bw} {\mathbf w}

\newcommand{\bG} {\mathbf G}

\newcommand{\cB} {\mathcal B}

\newcommand{\cF} {\mathcal F}
\newcommand{\cG} {\mathcal G}
\newcommand{\cH} {\mathcal H}

\newcommand{\cM} {\mathcal M}

\newcommand{\cP} {\mathcal P}

\newcommand{\Aut}{{{\operatorname{Aut}}}}
\newcommand{\Out}{{{\operatorname{Out}}}}

\newcommand{\Inn}{{{\operatorname{Inn}}}}
\newcommand{\Hom}{{{\operatorname{Hom}}}}

\newcommand{\Irr}{{{\operatorname{Irr}}}}

\newcommand{\GL}{\operatorname{GL}}

\newcommand{\Sp}{\operatorname{Sp}}

\newcommand{\PSL}{\operatorname{PSL}}
\newcommand{\SL}{\operatorname{SL}}

\newcommand{\Syl}{{{\operatorname{Syl}}}}

\newtheorem{thm}{Theorem}[section]
\newtheorem{lem}[thm]{Lemma}

\newtheorem{prop}[thm]{Proposition}
\newtheorem{conj}[thm]{Conjecture}

\theoremstyle{definition}

\newtheorem{defn}[thm]{Definition}

\begin{document}

\title{Weights for $\ell$-local compact groups}

\author{Jason Semeraro}

\begin{abstract}
In this note, we initiate the study of $\cF$-weights for an $\ell$-local compact group $\cF$ over a discrete $\ell$-toral group $S$ with discrete torus $T$. Motivated by Alperin's Weight Conjecture for simple groups of Lie-type, we conjecture that when $T$ is the unique maximal abelian subgroup of $S$ up to $\cF$-conjugacy and every element of $S$ is $\cF$-fused into $T$, the number of weights of $\cF$ is bounded above by the number of ordinary irreducible characters of its Weyl group. By combining the structure theory of $\cF$ with the theory of blocks with cyclic defect group, we are able to give a proof of this conjecture in the case when $\cF$ is simple and $|S:T| =\ell$. We also propose and give evidence for an analogue of the height zero case of Robinson's Ordinary Weight conjecture in this setting. 
\end{abstract}

\keywords{fusion systems, weight conjectures, blocks, $p$-local compact groups}

\subjclass[2010]{20C33, 20D20, 20D06, 55R35}

\date{\today}

\maketitle

\pagestyle{myheadings}

\section{Introduction}

An \textit{$\ell$-compact group} is, broadly speaking, the $\ell$-local homotopy theoretic analogue of a compact Lie group with a `Weyl group' $W$ which is a $\mathbb{Z}_\ell$-reflection group. In \cite{KMS1}, the authors introduce and prove a version of Alperin's Weight Conjecture (AWC) for fusion systems associated to homotopy fixed point spaces of connected $\ell$-compact groups under the action of unstable Adams operations. For such a fusion system $\cF$ associated to a simply connected $\ell$-compact group with Weyl group $W$, under some mild hypotheses this conjecture asserts that $\bw(\cF)=|\Irr(W)|$ where $\bw(\cF)$ is the number of \textit{weights} (defect zero characters of $\cF$-automorphism groups of $\cF$-centric radical subgroups up to $\cF$-conjugacy). One consequence is that the number of weights appears to be an invariant of the underlying $\ell$-compact group, as opposed to the space of fixed points. In an attempt to understand this, we introduce the study of weights for $\ell$-compact groups by appealing to the more general theory of \textit{$\ell$-local compact groups}  \cite{BLO3}. Here, we view an $\ell$-local compact group as a fusion system $\cF$ on an $\ell$-group $S$ which contains a finite index infinite torus $T$ (see Definition \ref{d:llcg}). Since $\cF$ has finitely many classes of $\cF$-centric radical subgroups, each with a finite $\cF$-outer automorphism group (see Proposition \ref{p:finite}), we obtain an integer invariant $\bw(\cF)$ which, by analogy with the finite case, we call the `number of weights' of $\cF$ (see Definition \ref{d:nrweights}). 

An $\ell$-local compact group $\cF$ is  \textit{connected} if $T$ is the unique maximal abelian subgroup of $S$ up to $\cF$-conjugacy and  every element of the underlying $\ell$-group is $\cF$-conjugate to an element of $T$. In this case we refer to the group $W=\Aut_\cF(T)$ as the \textit{Weyl group} of $\cF$. Based on \cite[Theorem 1]{KMS1} for finite fusion systems, and observations made in the present paper, we make the following conjecture which may be regarded as a weak analogue of Alperin's Weight Conjecture for connected $\ell$-local compact groups. 

\begin{conj}[Weak analogue of AWC]\label{c:main}
Let $\cF$ be a connected $\ell$-local compact group with Weyl group $W$. Then $$\bw(\cF) \le |\Irr(W)|.$$
\end{conj}

 Markus Linckelmann has asked for conditions on a fusion system $\cG$ which imply the existence of a subgroup $T$ with $\bw(\cG)=|\Irr(\Out_\cG(T))|$ so Conjecture \ref{c:main} would provide a partial answer to this.  Letting $v_\ell(-)$ denote the $\ell$-adic valuation, it is a straightforward observation that equality holds in Conjecture \ref{c:main} when $v_\ell(|W|)=0$. Indeed,  $\cF$ is automatically connected in this case (since $S=T$) and hence $T$ is the unique $\cF$-centric radical subgroup in $\cF$. Thus $\bw(\cF)$ is the number of defect zero characters of $W$, which is $|\Irr(W)|$ since $v_\ell(|W|)=0$. Our main result concerns the case $v_\ell(|W|)=1$.

\begin{thm}\label{t:main}
Suppose $\cF$ is a simple $\ell$-local compact group and $v_\ell(|W|)=1$. Then $\cF$ is connected and $$\bw(\cF)=|\Irr(W)|.$$
\end{thm}

Here, $\cF$ is \textit{simple} if it possesses no non-trivial normal subsystems. Theorem \ref{t:main} initially inspired the author to conjecture that an equality $\bw(\cF) = \Irr(W)$ \textit{always} holds for connected $\ell$-local compact groups, but this was later shown to be false in joint work with Kessar and Malle \cite{KMS3}. For example if $\cF=\cF_2(\Sp(2))$, then $4=\bw(\cF) \lneq |\Irr(D_8))|=5$ (see \cite[Example 4.2]{KMS3}). On the other hand, \cite[Theorem 1]{KMS3} shows that that $\bw(\cF)=|\Irr(W)|$ where $\cF=\cF_\ell(\bG)$ for \textit{any} compact connected Lie group $\bG$ for which the prime $\ell$ is good. Thus an equality-predicting refinement of Conjecture \ref{c:main} might be obtained from an appropriate generalisation of ``good'' to connected $\ell$-local compact groups; we do not pursue that here.

Next, we introduce a version of the height zero case of Robinson's Ordinary Weight Conjecture (OWC) for connected $\ell$-local compact groups. Via deep results in modular representation theory, Conjecture \ref{conj:amcmpt} below may be understood as a generalisation  of height zero OWC for Lie-type groups with Weyl group $W$ when $\ell$ is a very good prime and $q \equiv 1 \pmod \ell$. This connection is made formally at the end of Section \ref{sec:weights}.

\begin{conj}[Analogue of height zero OWC]\label{conj:amcmpt}
Let $\cF$ be a connected $\ell$-local compact group on $S$ with Weyl group $W$ and set $$\begin{array}{rcl} \cM(\cF) &=&\{(\psi,\chi) \mid \psi \in \Irr(S^{ab}), \chi \in \Irr(C_{\Out_\cF(S)}(\psi))\}/\sim_\cF \medskip \\ \cP(\cF) &=&\{(s,\chi) \mid s \in Z(S), \chi \in \Irr_0(W(s))\}/\sim_\cF\end{array}$$  where for each $s \in Z(S)$, $W(s)$ denotes the Weyl group of $C_\cF(s)$.  Then there is an $\Out(\cF)$-equivariant bijection $$\cM(\cF) \longleftrightarrow \cP(\cF).$$ 
\end{conj} 

Here $\Out(\cF)$ denotes the quotient $\Aut(\cF)/\Aut_\cF(S)$, where $\Aut(\cF) \le \Aut(S)$ is the group of $\cF$-fusion preserving automorphisms of $S$. The above equivalences $\sim_\cF$ are defined precisely in Section \ref{sec:weights}, as is the natural action of $\Out(\cF)$ on the equivalence classes $\cM(\cF)$ and $\cP(\cF)$. The conjectured existence of an $\Out(\cF)$-equivariant bijection is partially inspired by the \textit{inductive McKay condition } for simple groups due to Isaacs--Malle--Navarro \cite{IMN07}; indeed it seems likely that the natural map $\Out(G) \rightarrow \Out(\cF)$ (see \cite[Section 4]{AOV12}) can be used to exhibit a precise connection with that condition when $\cF$ is realisable by a simple group $G$, but we do not attempt that here. As with Conjecture \ref{c:main}, it is relatively straightforward to prove Conjecture \ref{conj:amcmpt} when $v_\ell(|W|)=0$ (see Proposition \ref{p:conj13ab}). Our second main result supplies some evidence for the case $v_\ell(|W|)=1$. We first recall the fact, due to Gonz\'{a}lez \cite[Theorem 1]{G16} that any $\ell$-local finite group $\cF$ on $S$ can be approximated as a direct limit of a sequence of (categorical) inclusions of saturated fusion systems on finite $\ell$-groups $$(\cF_1,S_1) \rightarrow (\cF_2,S_2) \rightarrow (\cF_3,S_3) \rightarrow \cdots  $$ with $S=\displaystyle\lim_{\substack{\longrightarrow \\ n \ge 1}} S_n$. We use this to establish a `truncated' version of 
Conjecture \ref{conj:amcmpt} in the case $v_\ell(|W|)=1$.

\begin{thm}\label{t:main2}
Let $\cF$ be a simple $\ell$-local compact group with Weyl group $W$ and assume $v_\ell(|W|)=1$. Then there is a sequence of fusion system inclusions $(\cF_n,S_n)_{n \ge 1}$ for which $\cF=\displaystyle\lim_{\substack{\longrightarrow \\ n \ge 1}} \cF_n$ where for each $n \ge 1$, $\cF_n$ is a finite fusion system on the  $\ell^n$-torsion subgroup  $S_n$ of $S$. Moreover, for each $n \ge 1$, there is a bijection $$\cM(\cF_n) \longleftrightarrow  \cP(\cF_n).$$ 
\end{thm}

We prove Theorems \ref{t:main} and \ref{t:main2} in Section \ref{s:main2}. The proofs only require  the character theory of finite groups with Sylow $\ell$-subgroups of order $\ell$ and particular information concerning the $\cF$-automorphisms of $\cF$-centric radical subgroups mostly established in \cite{OR19} (see Theorems \ref{t:or2} and \ref{t:or1}). The connectedness of $\cF$ is crucial to  ensuring that there are a sufficiently small number of classes of $\cF$-centric radical subgroups.

We expect Theorem \ref{t:main2} to follow from Conjecture \ref{conj:amcmpt}. In \cite{JLL12}, Junod, Levi and Libman  prove that for any $\ell$-local compact group $\cF$ with torus $T$ and any $\zeta \in \ZZ_\ell^\times$ with $v_\ell(\zeta-1)  >\!\!>0$, $\Aut(\cF)$ contains an \textit{unstable Adams operation} $\Psi=\Psi(\zeta)$ with the property that $\Psi|_T(g)=g^\zeta$ for each $g \in T$. Given such an operation $\Psi$, computations of Gonz\'{a}lez \cite{G12} seem to indicate that the homotopy fixed points of powers of $\Psi$ stratify $\cF$ as a direct limit of finite fusion systems $\{\cF_n \mid n \ge 1 \}$ where $\cF_n$ is the fusion subsystem of $\cF$ fixed by $\Psi^n$  (so $\cF=\displaystyle\lim_{\substack{\longrightarrow \\ n \ge 1}} \cF_n$). Thus the $\Aut(\cF)$-equivariance in  Conjecture \ref{conj:amcmpt} should imply the existence of bijections between sets fixed by $\Psi^n \in \Aut(\cF)$ as in Theorem \ref{t:main2}.

\subsection{Acknowledgements}
 The author would like to thank the Isaac Newton Institute for Mathematical Sciences, Cambridge, for support and hospitality during the programme ``Groups, representations and applications: new perspectives'' (supported by EPSRC grant EP/R014604/1.\#34) where work on this paper was undertaken. Thanks in particular are extended to Gunter Malle, Radha Kessar and Bob Oliver for helpful discussions. The author also gratefully acknowledges personal support from the Heilbronn Institute for Mathematical Research and the EPSRC (grant EP/W028794/1). Finally the author would like to thank the anonymous referee for carefully reading the paper and for making numerous suggestions which led to its improvement.

\section{$\ell$-local compact groups and weights}\label{sec:weights}
For the rest of the paper, $\ell$ is a prime and $k$ is an algebraically closed field of characteristic $\ell$. We refer the reader to \cite{AKO11} for basic notation and results in the theory of fusion systems.  We begin by introducing the class of $\ell$-groups on which we focus our attention.

\begin{defn}
Write $\ZZ/\ell^\infty$ for the infinite union of the cyclic $\ell$-groups $\ZZ/\ell^n$ under the obvious inclusions. A group isomorphic to a direct product of finitely many copies of $\ZZ/\ell^\infty$ is called a \textit{discrete $\ell$-torus}. \textit{A discrete $\ell$-toral group} is a group $S$ which contains a discrete $\ell$-torus $T$ as a normal subgroup of finite $\ell$-power index.  
\end{defn}

As observed in \cite[Lemma 1.3]{BLO3}, any quotient, subgroup or extension of an infinite discrete $\ell$-toral group $S$ is a discrete $\ell$-toral group and $S$ contains only finitely many conjugacy classes of subgroups of order $\ell^n$ for every $n \ge 0$ (\cite[Lemma 1.4]{BLO3}). By \cite[Lemma 1.9]{BLO3} there is an increasing sequence $S_1 \le S_2 \le S_3 \le \cdots$ of $\Aut(S)$-invariant $\ell^n$-torsion subgroups of $S$, whose direct limit is $S=\displaystyle\lim_{\substack{\longrightarrow \\ n \ge 1}} S_n$. When $S$ is abelian, the sequence $$\cdots \rightarrow \Hom(S_3,\CC^\times) \rightarrow \Hom(S_2,\CC^\times) \rightarrow \Hom(S_1,\CC^\times) \rightarrow \Hom(S_1,\CC^\times)$$ is $\Aut(S)$-invariant, so that $\Aut(S)$ acts naturally on the corresponding inverse limit $\Irr(S) = \displaystyle\lim_{\substack{\longleftarrow \\ n \ge 1}} \Hom(S_n,\CC^\times)$.

 A \textit{fusion system} $\cF$ over $S$ with discrete torus $T$ is defined by analogy with the finite case \cite[Definition 2.1]{BLO3}. What it means for $\cF$ to be \textit{saturated} is almost analogous, provided one makes an additional assumption that morphisms behave well with respect to filtrations of subgroups (see \cite[Definition 2.2 (III)]{BLO3}). 

\begin{defn}\label{d:llcg}
Let $S$ be a discrete $\ell$-toral group with discrete torus $T$. An \textit{$\ell$-local compact group} is a saturated fusion system on $S$. $\cF$ is, in addition, \textit{connected} if $T$ is the unique maximal abelian subgroup of $S$ up to $\cF$-conjugacy and every element of $S$ is $\cF$-conjugate to an element of $T$. In this case we refer to the group $\Aut_\cF(T)$ as the \textit{Weyl group} of $\cF$.
\end{defn}

Note that the definition of \textit{connected} here differs slightly to that found in \cite[Definition 3.1.4]{G10}. The definitions of \textit{fully $\cF$-centralised}, \textit{fully $\cF$-normalised}, \textit{$\cF$-centric} and \textit{$\cF$-radical} are unchanged from those for fusion systems on finite groups. We write $\cF^{cr}$ for the set of $\cF$-centric radical subgroups. As in \cite[Definition 1.4]{OR19}, we say that $\cF$ is \textit{simple} if it possesses no non-trivial normal subsystems. Using a certain ``bullet construction'' first considered by Benson (see \cite[Definition 3.1]{BLO3}) it is possible to prove the following result.

\begin{prop}\label{p:finite}
Let $\cF$ be an $\ell$-local compact group on $S$. Then $\Out_\cF(P)$ is finite for all $P \le S$. Moreover, $\cF$ has only finitely many conjugacy classes of $\cF$-centric, $\cF$-radical subgroups.
\end{prop}

\begin{proof}
See \cite[Proposition 2.3 and Corollary 3.5]{BLO3}.
\end{proof}

In particular the Weyl group of $\cF$ is finite. We are interested in the following invariant.

\begin{defn}\label{d:nrweights}
Let $\cF$ be an $\ell$-local compact group on $S$. The \textit{number of weights} $\bw(\cF)$ of $\cF$ is given by: $$\bw(\cF):=\sum_{P \in \cF^{cr}/\cF} z(k\Out_\cF(P)),$$ where the sum runs over a set of $\cF$-conjugacy class representatives of $\cF$-centric, $\cF$-radical subgroups and for a group $H$, $z(kH)$ denotes the number of projective simple $kH$-modules up to isomorphism.
\end{defn}

By Proposition \ref{p:finite}, $\bw(\cF)$ is finite. For the remainder of this section we assume that $\cF$ be a \textit{connected} $\ell$-local compact group on $S$.

\begin{defn}\label{def:ordinaryweights}
The set $\cM(\cF)$ of \textit{ordinary weights of height $0$} is the set of equivalence classes $$\cM(\cF)=\{(\psi,\chi) \mid \psi \in \Irr(S^{ab}), \chi \in \Irr(C_{\Out_\cF(S)}(\psi)) \}/\sim_\cF ,$$  where $(\psi, \chi) \sim_\cF (\psi',\chi')$ if there is $\rho \in \Aut_\cF(S)$ with $\psi^\rho=\psi'$ and $\chi^{\rho^{-1}}=\chi'$; we write $[(\psi,\chi)]$ for the class of $(\psi,\chi)$. 
\end{defn}

Recall that an element $\phi \in \Aut(S)$ is an \textit{automorphism of $\cF$} if it induces an autofunctor $\Phi: \cF \rightarrow \cF$ where:
\begin{enumerate}
\item for each $P \in$ Ob$(\cF)$, $\Phi(P)=\phi(P)$;
\item for each $\psi \in \Hom_\cF(P,Q)$, $$\Phi(\psi) = \phi|_{Q} \circ \psi \circ (\phi|_{\phi(P)})^{-1} : \phi(P) \rightarrow \phi(Q)$$ is an element of $\Hom_\cF(\phi(P),\phi(Q))$.
\end{enumerate}

 Write $\Aut(\cF) \le \Aut(S)$ for the set of all automorphisms of $S$ which induce automorphisms of $\cF$. Since $\Aut_\cF(S) \unlhd \Aut(\cF)$, we define $\Out(\cF):=\Aut(\cF)/\Aut_\cF(S)$ to be the  \textit{outer automorphism group of $\cF$} and denote by $[\phi] \in \Out(\cF)$ the class  of an element $\phi$ of $\Aut(\cF)$. If $s \in Z(S)$ then $s$ is fully $\cF$-centralised, so $C_\cF(s)$ is also a saturated fusion system on $S$ and we write $W(s):=C_W(s)$. 
 
 \begin{lem}\label{l:cent}
For any $s \in Z(S)$ and $\phi \in \Aut(\cF)$, the induced autofunctor $\Phi$ restricts to an isomorphism $C_\cF(s) \rightarrow C_\cF(\phi(s))$ which sends $W(s)$ to $W(\phi(s))$.
 \end{lem}

\begin{proof}
Let $P \le S$, and $\psi \in \Hom_{C_\cF(s)}(P,S)$, then by definition of $C_\cF(s)$, $\psi$ extends to a map $\overline{\psi} \in \Hom_\cF(P\langle s \rangle,S)$ which fixes $s$; hence $$\phi \circ \overline{\psi} \circ \phi^{-1} \in \Hom_\cF(\phi(P\langle s \rangle),S) = \Hom_\cF(\phi(P) \langle \phi(s) \rangle,S)$$ and this latter map clearly fixes $\phi(s)$ and  restricts to  $\phi \circ \psi \circ \phi^{-1} \in \Hom_\cF(\phi(P),S)$. We deduce that $\phi \circ \psi \circ \phi^{-1} \in \Hom_{C_\cF(\phi(s))}(\phi(P),S)$. Moreover, since $\Phi$ fixes $T$, $\phi|_T \circ  W \circ (\phi|_T)^{-1} = W$ and hence $\phi|_T  W(s) (\phi|_T)^{-1} = W(\phi(s))$.
\end{proof}

\begin{defn}
 The set $\cP(\cF)$ is the set of equivalence classes $$\cP(\cF):=\{(s,\chi) \mid s \in Z(S), \chi \in \Irr_0(W(s))\}/\sim_\cF,$$ where $(s,\chi) \sim_\cF (s',\chi')$ if and only if there is $\phi \in \Aut_\cF(S)$ with $s'=\phi(s)$ and $\chi'=\chi^{\Phi^{-1}|_{W(\phi(s))}}$. 
\end{defn}

\begin{prop}\label{p:nataction}
$\Out(\cF)$ acts on $\cP(\cF)$ via $$[\phi] \cdot [(s,\chi)] = [(\phi(s), \chi^{\Phi^{-1}|_{W(\phi(s))}})],$$ for each $[\phi] \in \Out(\cF)$ and class $[(s,\chi)] \in \cP(\cF)$. \end{prop}

\begin{proof}
It suffices to check the action of $[\phi]$ on $\chi$ is independent of the choice of representative $\phi$. If $\phi' \in \Aut(\cF)$ is such that $[\phi']=[\phi]$ then $\gamma \circ \phi'=\phi$ for some $\gamma \in \Aut_\cF(S)$ so $$[\phi'] \cdot [(s,\chi)] = [(\phi'(s), \chi^{\Phi'^{-1}})] \sim_\cF [(\phi(s), \chi^{\Phi^{-1}})]$$ since  $(\gamma \circ \phi')(s)=\phi(s)$ and $\chi^{\Phi'^{-1}\Gamma^{-1}}=\chi^{\Phi^{-1}}$ where $\Gamma$ is the functor $C_\cF(\phi'(s)) \rightarrow C_\cF(\phi(s))$ induced by $\gamma$.
\end{proof}

Similarly, $\Out(\cF)$ acts on $\cM(\cF)$ via $[\phi] \cdot [(\psi,\chi)] = [(\psi^\phi, \chi^{\phi^{-1}}]$ for each $[\phi] \in \Out(\cF)$ and class $[(\psi,\chi)] \in \cM(\cF)$ and thus the statement of Conjecture \ref{conj:amcmpt} makes sense. In particular we can prove:

\begin{prop}\label{p:conj13ab}
Conjecture \ref{conj:amcmpt} holds if $S$ is abelian.
\end{prop}

\begin{proof}
Since $|W|$ is coprime with $\ell$, by the Glauberman correspondence, there exists an infinite family of $W$-equivariant bijections $S_n \rightarrow \Irr(S_n)$ which, by the remarks following Definition \ref{d:llcg}, induces a $W$-equivariant bijection $S \rightarrow \Irr(S)$. The result follows immediately from this.
\end{proof}

We conclude this section by explaining the sense in which Conjecture \ref{conj:amcmpt} may be regarded as an analogue of OWC for height zero characters. Let $B$ be the principal $\ell$-block of a finite group $G$, $S \in \Syl_\ell(G)$ be its defect group and $\cF=\cF_S(G)$ be its fusion system. The maximal defect case of OWC for $B$ (see  \cite[Section 2]{KLLS19}) predicts an equality \begin{equation}\label{e:mfd}
\bm(\cF,d)=\Irr^d(B),\end{equation} for $d=v_\ell(|S|)$ where $\Irr^d(B)$ denotes the set of characters of $\ell$-defect $d$ in $B$. We have \begin{equation}\label{e:mfd2} \bm(\cF,v_\ell(|S|))=\sum_{\chi \in \Irr_0(S)/\Out_\cF(S)} |\Irr(C_{\Out_\cF(S)}(\chi))|=|\Irr(S^{ab}:\Out_\cF(S))|= |\cM(\cF)|,\end{equation} where the last equality follows from Lemma \ref{l:molg}. If $G$ is of Lie-type we have the following result.

\begin{lem}\label{l:goltmot}
Suppose that $B=B_0(G)$ is the principal $\ell$-block of $G=\bG^F$ with defect group $S$, where $\bG$ is a semisimple algebraic group defined over $\FF_q$ (for which $\ell$ is very good) with respect to a Frobenius endomorphism $F: \bG \rightarrow \bG$ and let $\cF$ be the $\ell$-fusion system of $G$ on $S$. If $q \equiv 1 \pmod \ell$ then $|\Irr^d(B)|=|\cP(\cF)|$.
\end{lem}

\begin{proof}
The proof of \cite[Proposition 6.8]{KMS1} shows that there is a degree-preserving bijection between $\Irr(B_0(G))$ and the set of irreducible characters in the principal $\ell$-block $B_0$ of the associated $\ZZ_\ell$-spets $\GG(q)$ (see \cite[Definition 6.7]{KMS1}). The number of characters of defect $d$ in the latter set is computed in \cite[Proposition 6.11]{KMS1}. In the notation of that result, we see that for each $s \in S$, $u_s=0$ by \cite[Proposition 5.7]{KMS1} and $W(s)_{\phi_s^{-1}\zeta^{-1}}=W(s)$ since $q \equiv 1 \pmod \ell$. Since maximal defect characters only occur when $s \in Z(S)$, the result follows.
\end{proof}

Thus, in the situation of Lemma \ref{l:goltmot}, by (\ref{e:mfd2}) we see that height zero OWC for $B$ is exactly the equality $|\cM(\cF)|=|\cP(\cF)|$.

\section{Some character theory}\label{s:prelim}

We need an elementary consequence of the character theory of groups with a Sylow $\ell$-subgroup of order $\ell$.

\begin{lem}\label{l:thev}
Let $W$ be a finite group with a cyclic Sylow $\ell$-subgroup $U$ of order $\ell$. Then $$|\Irr(W)|-z(kW)=|\Irr(N_W(U))|=|\Irr_0(N_W(U))| = |\Irr_0(W)|.$$
\end{lem}

\begin{proof}
This follows from the existence of a bijection between blocks of $W$ and $N_W(U)$  with non-trivial defect which preserves the number of irreducible characters (see \cite[VII.2.12]{F82}). Hence summing over all such blocks we obtain, $$|\Irr(W)|-z(kW)=|\Irr(N_W(U))|-z(kN_W(U))=|\Irr(N_W(U))|,$$ since $z(kN_W(U))=0$, as required.

\end{proof}

The following is sometimes referred to as the ``Method of Little Groups.''

\begin{lem}\label{l:molg}
Let $H \unlhd G$ be finite groups and $\Theta:=\Irr(H)/G$ be a set of $G$-orbit representatives for $\Irr(H)$. Suppose that each $\theta \in \Irr(H)$ extends to a character $\widehat{\theta} \in \Irr(I_G(\theta))$, where $I_G(\theta)$ denotes the inertia subgroup of $\theta$ in $G$. Then there is a bijection  $$\{(\theta,\beta) \mid \theta \in \Theta, \beta \in \Irr(I_G(\theta)/H)\} \longrightarrow \Irr(G), \mbox{ given by } (\theta,\beta) \mapsto (\hat{\theta}\beta)_{I_G(\theta)}^G.$$
\end{lem}

\begin{proof}
See, for example, \cite[Theorem 11.5]{CR90}.
\end{proof}

The next lemma concerning characters of the extension of a normal subgroup by an $\ell'$-group is well-known.

\begin{lem}\label{l:klls}
Let $G$ be a finite group with $N \unlhd G$. Suppose that $V$ is an inertial projective simple $kN$-module and that $G/N$ is a cyclic $\ell'$-group. Then $G$ has exactly $|G:N|$ projective simple modules lying over $V$. In particular if $N=\SL_2(\ell) \le G \le \GL_2(\ell)$, then $z(kG)=|G:\SL_2(\ell)|$.
\end{lem}

\begin{proof}
See, for example, \cite[Lemma 4.12]{KLLS19}.
\end{proof}

The following technical result is used in our calculation of $\bw(\cF)$ in Section \ref{s:main}.

\begin{lem}\label{l:chars}
Let $X_1$ be a finite $\ell'$-group and $X_2$ be a finite group and suppose $H$ is a normal subgroup of $X_1 \times X_2$ with $(X_1 \times X_2)/H$ cyclic of order $e \mid \ell-1$. The following hold:
\begin{itemize}
\item[(1)] If $X_2 = \GL_2(\ell)$ and $O^{\ell'}(H)=\SL_2(\ell)$ then $z(kO^{\ell'}(H)X_1)=|\Irr(X_1)|\cdot (\ell-1)/e$. 
\item[(2)] If $X_2 = N_{\GL_2(\ell)}(U)$ for some Sylow $\ell$-subgroup $U$ of $\SL_2(\ell)$ and $O^{\ell'}(H) = U$ then $|\Irr(H)|=|\Irr(X_1)|\cdot \ell \cdot (\ell-1)/e$.
\end{itemize}
\end{lem}

\begin{proof}
Suppose the hypotheses of (1) hold. Since $X_1$ is an $\ell'$-group,  $z(k(X_1 \times X_2))=|\Irr(X_1)|\cdot (\ell-1)$ by Lemma \ref{l:klls}. On the other hand, Lemma \ref{l:klls} applied to $N=O^{\ell'}(H)X_1$ and $G=X_1 \times X_2$ yields $z(k(X_1 \times X_2))=z(kO^{\ell'}(H)X_1)\cdot e$ and (1) follows from this. Part (2) is proved similarly: in this case $X_1 \times X_2 \cong X_1 \times (C_\ell \rtimes (C_{\ell-1})^2)$, and this group has $|\Irr(X_1)|\ell(\ell-1)$ irreducible characters. Now Lemma \ref{l:molg} implies that $|\Irr(X_1)| \ell(\ell-1)=|\Irr(H)|e$, as needed.
\end{proof}

\section{Fusion systems on $\ell$-groups with an abelian maximal subgroup}\label{s:main}

In this section we assume that $\cF$ is a simple $\ell$-local compact group on $S$ with discrete torus $T$ of index $\ell$. We adopt the  notation of \cite{OR19} to describe the structure of $\cF$. Set $$\cH=\{Z(S)\langle x \rangle \mid x \in S \backslash T\} \mbox{ and } \cB=\{Z_2(S)\langle x \rangle \mid x \in S \backslash T\},$$ where $Z(S)$ and $Z_2(S)$ denote the first and second centres of $S$ respectively. We have the following result from \cite{OR19} when $T$ is not $\cF$-centric radical.

\begin{thm}\label{t:or2}
Assume that $T$ is not $\cF$-centric radical. Then $\cF$ is connected, $T$ has rank $\ell-2$, $\cF^{cr}=\{S\} \cup \cH$ and for any prime $q \neq \ell$, $\cF$ is isomorphic to the unique subsystem of the $\ell$-fusion system of $\Gamma=\PSL_\ell(\overline{\mathbb{F}_q})$ with
\begin{itemize}
\item[(1)] $\Out_\cF(S)=\Out_\Gamma(S) \cong C_{\ell-1}$;
\item[(2)] $\Aut_\cF(T)=\Aut_{N_\Gamma(S)}(T) \cong C_\ell \rtimes C_{\ell-1}$, where $C_{\ell-1}$ is acting faithfully; and 
\item[(3)] $\Aut_\cF(P)=\Aut_{\Gamma}(P) \cong \SL_2(\ell)$, for $P \in \cH$.
\end{itemize}
\end{thm}

\begin{proof}
Parts (1), (2) and (3) follow from \cite[Theorem 5.12]{OR19}. If $x \notin T$ then $x$ is $\Aut_\cF(P)$-conjugate to an element of $Z(S) \le T$ where $P=Z(S)\langle x \rangle \in \cH$ and $\cF$ is connected.
\end{proof}

In particular, for $\cF$ as described in Theorem \ref{t:or2}, we have $\cH=P^\cF$ for any $P \in \cH$, and Lemma \ref{l:chars}(1) applied in the case $X_1 = 1$ and  $H=\SL_2(\ell)$ yields $$\bw(\cF)=z(k\Out_\cF(S))+z(k\Out_\cF(P))=(\ell-1)+z(k\SL_2(\ell))=\ell=|\Irr(\Aut_\cF(T))|,$$
 so Theorem \ref{t:main} holds in this case. Thus for the remainder of this subsection we may assume that $T$ is an $\cF$-centric radical subgroup of $S$.

The following notation from \cite{OR19} provides a convenient way to describe the actions of $\Out_\cF(S)$ and $\Out_\cF(T)$ on $S/T$ and $Z(S) \cap [S,S]$, as well as the two-way traffic between them. Define: $$\Aut^\vee(S):=\{\alpha \in \Aut(S) \mid [\alpha, Z(S)] \le Z(S) \cap [S,S] \} \mbox{ and } \Aut^\vee(T) = \{\alpha|_T \mid \alpha \in \Aut^\vee(S) \}$$
and let
$$\Aut^\vee_\cF(S):=\Aut^\vee(S) \cap \Aut_\cF(S) \mbox{ and } \Aut^\vee_\cF(T):=\Aut^\vee(T) \cap \Aut_\cF(T).$$ Hence, $$\Aut_\cF^\vee(T)=\{\beta \in N_{\Aut_\cF(T)}(\Aut_S(T)) \mid [\beta,Z(S)] \le Z(S) \cap [S,S]\}$$
(see \cite[Notation 2.9]{OR19}). Observe that $\Inn(S) \le  \Aut_\cF^\vee(S)$ and $\Aut_\cF^\vee(S) \unlhd \Aut_\cF(S)$ (it is the kernel of a homomorphism to $\Aut(Z(S)/(Z(S) \cap [S,S]))$ so we may define $\Out_\cF^\vee(S):=\Aut_\cF^\vee(S)/\Inn(S).$ Now set, $$\Delta:=(\ZZ/\ell)^\times \times (\ZZ/\ell)^\times, \mbox{ and } \Delta_n:=\{(r,r^i) \mid r \in (\ZZ/\ell)^\times\} \le \Delta \mbox{ for $i \in \ZZ$ }, $$
define $$\mu: \Aut_\cF^\vee(S) \longrightarrow \Delta \mbox{ via } \mu(\alpha)=(r,s) \mbox{ if } \begin{cases} x\alpha \in x^rT & \mbox{ for $x \in S \backslash T$ } \\ g\alpha = g^s & \mbox{ for $ g \in Z(S) \cap [S,S],$} \end{cases}$$
and let $$\mu_T: \Aut_\cF^\vee(T) \longrightarrow \Delta \mbox{ and } \hat{\mu}: \Out_\cF^\vee(S) \longrightarrow \Delta$$ be given by $\mu_T(\alpha|_T)=\mu(\alpha)$ and $\hat{\mu}([\alpha])=\mu(\alpha)$ if $\alpha \in \Aut_\cF^\vee(S)$ (here $[\alpha]$ denotes the class of $\alpha$ in $\Out_\cF^\vee(S)$).

 We begin by appealing to results in \cite{COS17, OR19} to describe particular groups of automorphisms of subgroups which turn out to be $\cF$-centric radical. The following result is crucial to our analysis.
 
\begin{lem}\label{l:O14adapt}
Let $A$ be an abelian discrete $\ell$-toral group and $G \le \Aut(A)$ be finite. Assume $v_\ell(|G|)=1$, $O_\ell(G)=1$ and that $|[x,A]|=\ell$ for every element $x \in G$ of order $\ell$. Then, there  is a factorization $A=A_1 \times A_2$ with $A_2 \cong C_\ell \times C_\ell$ such that $\Aut_G(A_2)$ contains $\SL_2(\ell)$ and $G$ is a normal subgroup of index dividing $\ell-1$ in $\Aut_G(A_1) \times \Aut_G(A_2)$.
\end{lem}

\begin{proof}
See \cite[Proposition A.7]{OR19}.
\end{proof}
 
We first deal with the subgroups $Q \in \cB \cup \cH$ described above.

\begin{lem}\label{l:outfp}
Let $Q \in \cB \cup \cH$. There are unique subgroups $\tilde{Z} \le Z(S)$ and $\tilde{Q} \ge Q \cap [S,S]$ such that 
\begin{itemize}
\item[(1)] if $Q \in \cH$ then $Q=\tilde{Z} \times \tilde{Q}$ and $\tilde{Q} \cong C_\ell \times C_\ell$; and
\item[(2)] if $Q \in \cB$ then $\tilde{Z}=Z(S)$, $Q=\tilde{Z}\tilde{Q}$ and $\tilde{Q} \cong \ell^{1+2}_+$.
\end{itemize}
In either case if $Q$ is $\cF$-centric radical there is a unique subgroup $\Theta \le \Aut(Q)$ containing $\Inn(Q)$ which acts trivially on $\tilde{Z}$, normalises $\tilde{Q}$ and is such that $\Theta/\Inn(Q) \cong \SL_2(\ell)$.  Moreover, there is an $\ell'$-subgroup $X$ of $\Out_\cF(Q)$ such that $$\Out_\cF(Q) = X(\Theta/\Inn(Q)) \mbox{ and } N_{\Out_\cF(Q)}(\Out_S(Q)) = X(C_{\ell} \rtimes C_{\ell-1}).$$ 

\end{lem}

\begin{proof}
Parts (1) and (2) are shown in \cite[Lemma 5.9]{OR19}. The existence of the subgroup $\Theta$ with the stated properties is argued in the proof of \cite[Theorem 5.11]{OR19} where it is also shown that $\Theta$ has $\ell'$-index in $\Aut_\cF(Q)$. From Lemma \ref{l:O14adapt} applied with $Q$  (case (1)) or $Q/(Z(S) \cap [S,S])$ (case (2)) and $\Theta$ in the roles of $A$ and $G$ respectively we have $$X_1(\Theta/\Inn(Q)) \le \Out_\cF(Q) \le X_1 \times X_2$$ where $X_2 \cong \GL_2(\ell)$ (see also \cite[Lemma 5.9(b)(iv)]{OR19}). This description of $\Out_\cF(Q)$ implies that there must exist a group $X$ with the stated properties. The lemma follows.
\end{proof}

We next deal with automorphisms of $S$ and $T$.

\begin{lem}\label{l:outftands}
Suppose $Q \in \cB \cup \cH$ is $\cF$-centric radical and $X$ is as in Lemma \ref{l:outfp}. For some $t \in \{0,-1\}$ the following hold:
\begin{itemize}
\item[(1)] There is an $\ell'$-subgroup $Y \le \Out_\cF(S)$ with $Y \cong X$ such that $$\Out_\cF(S) =  Y \times  \Out_\cF^\vee(S) \cap \hat{\mu}^{-1}(\Delta_t),$$ where $\Out_\cF^\vee(S) \cap \hat{\mu}^{-1}(\Delta_t) \cong C_{\ell-1}$.
\item[(2)] There is an $\ell'$-subgroup $Z \le N_{\Aut_\cF(T)}(\Aut_S(T))$ where $Z \cong X$ such that  $$N_{\Aut_\cF(T)}(\Aut_S(T)) = Z( \Aut_\cF^\vee(T) \cap \mu_T^{-1}(\Delta_t)),$$ with $\Aut_\cF^\vee(T) \cap \mu_T^{-1}(\Delta_t)  \cong C_\ell \rtimes C_{\ell-1}$.
\end{itemize}
\end{lem}

\begin{proof}
By \cite[Corollary 2.6]{OR19}, we may choose $x \in S \backslash T$ of order $\ell$ and $Q \in \cB \cup \cH$ so that $x \in Q$. Plainly, each element of $\Aut_\cF(S)$ normalises $S/T \cong \langle x \rangle$, so in particular the restriction to $Q$ of any $\ell'$-element in $\Aut_\cF(S)$ lies in $N_{\Aut_\cF(Q)}(\Aut_S(Q))$. By the Schur--Zassenhaus theorem, $\Inn(S)$ and $\Aut_S(Q)$ are complemented by $\ell'$-groups $Y_1$ and $Y_2$ in $\Aut_\cF(S)$ and $N_{\Aut_\cF(Q)}(\Aut_S(Q))$ respectively. We claim that the restriction map $\psi: \Aut_\cF(S) \longrightarrow N_{\Aut_\cF(Q)}(\Aut_S(Q))$ restricts to an isomorphism from $Y_1$ to $Y_2$. If $\alpha \in \ker(\psi|_{Y_1})$ then $\alpha$ acts trivially on $Q$ so $\alpha \in \Aut_\cF^\vee(S)$, $\hat{\mu}([\alpha])=(1,1)$ and  hence $[\alpha]=1$ since $\hat{\mu}|_{\Out_\cF^\vee(S)}$ is injective (by \cite[Lemma 4.3(a)]{OR19}). We conclude that $\alpha=1$ whence $\psi|_{Y_1}$ is injective.  Conversely, since $\cF$ is saturated, every element of $Y_2$ must extend to an element of $Y_2$ and $\psi|_{Y_1}$ is also surjective. We conclude that $Y_1 \cong Y_2$.  By \cite[Lemma 5.9(b)(iii)]{OR19}, the restriction to $Q$ of elements in $\Aut_\cF^\vee(S) \cap \mu^{-1}(\Delta_t)$ lies in $N_{\Theta}(\Aut_S(Q)) \cong C_\ell:C_{\ell-1}$. Hence, by Lemma \ref{l:outfp}, we must have $\Out_\cF(S)=Y \times \Out_\cF^\vee(S) \cap \widehat{\mu}^{-1}(\Delta_t)$ where $$X \cong Y:=\{\alpha \in \Out_\cF(S) \mid [\alpha,Z(S)] \nleq Z(S) \cap [S,S]]\},$$ 
(with $X$ is as in Lemma \ref{l:outfp}). Hence (1) holds. We next prove (2). Since every element of $N_{\Aut_\cF(T)}(\Aut_S(T))$ extends to an automorphism of $S$, the description of this group follows from (1) with $Z=\{\alpha|_T \mid \alpha \in Y\}$. Moreover, $\Aut_S(T) \le \Aut_\cF^\vee(T), \mu_T^{-1}(\Delta_t)$ since $[x,Z(S)] \le Z(S) \cap [S,S]$ for $x \in S \backslash T$ so $\Aut_\cF^\vee(T) \cap \mu_T^{-1}(\Delta_t) \cong C_\ell \rtimes C_{\ell-1}$ and (2) holds. 
\end{proof}

The structure of $\cF$ is now given as follows. 

\begin{thm}\label{t:or1}
Let $\cF$ be a simple $\ell$-local compact group on $S$ with discrete torus $T$ of index $\ell$ which is $\cF$-centric radical. Then $\cF$ is connected and for some $t \in \{0,-1\}$ one of the following holds:
\begin{itemize}
\item[(1)] $\dim(\Omega_1(T))=\ell-1$, $\mu_T(\Aut_\cF^\vee(T)) \ge \Delta_t$ and $\Aut_\cF(T)=O^{\ell'}(\Aut_\cF(T))\mu_T^{-1}(\Delta_t)$; 
\item[(2)] $\dim(\Omega_1(T)) \ge \ell$, $t=0$, $\mu_T(\Aut_\cF^\vee(T))=\Delta_0$ and $\Aut_\cF(T)=O^{\ell'}(\Aut_\cF(T))\Aut_\cF^\vee(T)$.
\end{itemize}
In either case, $$\cF^{cr}=\begin{cases} \{S,T\} \cup \cH & \mbox{ if $t=-1$ } \\ \{S,T\} \cup \cB & \mbox{ if $t=0$. } \end{cases} $$ where the sets $\cH$ and $\cB$ each form a single $S$-conjugacy class and there is an $\ell'$-group $X$ such that for each $Q \in \cH$ (if $t=-1$) or $Q \in \cB$ (if $t=0$), $$\Out_\cF(S) \cong C_{\ell-1} \times X, \hspace{2mm} N_{\Aut_\cF(T)}(\Aut_S(T)) \cong (C_\ell : C_{\ell-1})X, \hspace{2mm} \mbox{ and }\hspace{2mm}  \Out_\cF(Q) \cong \SL_2(p) X.$$
\end{thm}

\begin{proof}
Parts (1) and (2) and the description of $\cF^{cr}$ follow from \cite[Theorem B(a)]{OR19}. $\cH$ and $\cB$ each consist of one $S$-conjugacy class by \cite[Lemma 5.4]{OR19}. The descriptions of $\Out_\cF(Q)$, $\Out_\cF(S)$ and $N_{\Aut_\cF(T)}(\Aut_S(T))$ follow from more precise descriptions in Lemmas \ref{l:outfp} and \ref{l:outftands}. 
By \cite[Corollary 5.2]{OR19}, $T$ is the unique abelian subgroup of index $\ell$ in $S$, and moreover, any $x \in S \backslash T$ is $\cF$-conjugate to an element of $Q \cap [S,S] < T$ via an $\cF$-automorphism of $Q=Z(S)\langle x \rangle$ $(t=-1)$ or  $Q=Z_2(S)\langle x \rangle$ $(t=0)$. Hence $\cF$ is connected. 
\end{proof}

 We end this section by showing that $\cF$ can be expressed as direct limit of finite fusion systems. 

\begin{prop}\label{p:approx}
Let $\cF$ be a simple fusion system over over an infinite non-abelian discrete $\ell$-toral group $S$ with discrete torus $T$ of index $\ell$. Then $\displaystyle\cF = \lim_{\substack{\longrightarrow \\ n \ge 1}} \cF_n$ for saturated fusion systems $\cF_n$ on $S_n \le S$ with $\displaystyle S = \lim_{\substack{\longrightarrow \\ n \ge 1}} S_n$.
\end{prop}

\begin{proof}
If $T$ is not $\cF$-centric radical then for any prime $q$ with $v_\ell(q-1)=1$ the $\ell$-fusion system of $\PSL_\ell(\overline{\mathbb{F}_q})$ is a direct limit $$\lim_{\substack{\longrightarrow \\ n \ge 1}} \cF_{S_n}(\PSL_\ell(q^{\ell^{n-1}})), \mbox{ where }  \displaystyle S = \lim_{\substack{\longrightarrow \\ n \ge 1}} S_n \mbox{ with } S_n \in \Syl_\ell(\PSL_\ell(q^{\ell^{n-1}})).$$ Setting $\cF_n$ to be the unique subsystem of $\cF_{S_n}(\PSL_\ell(q^{\ell^{n-1}}))$ with  $$\Aut_{\cF_n}(S_n)= C_{\ell-1}, \hspace{2mm} \Aut_{\cF_n}(T_n) \cong C_\ell \rtimes C_{\ell-1} \mbox{ and } \Aut_{\cF_n}(P_n) \cong \SL_2(\ell) \mbox{ for $P_n \in \cH$},$$ we realise $\cF$ as a direct limit of fusion systems.

Now suppose $T$ is $\cF$-centric radical, and for each $P \in \cF^{fcr}$ and $n \in \NN$, let $P_n$ be the $\ell^n$-torsion subgroup of $P$. Then $\displaystyle P = \lim_{\substack{\longrightarrow \\ n \ge 1}} P_n$ and we may define $$\cF_n:=\langle \{\varphi|_{P_n} \mid \varphi \in \Out_\cF(P) \} \mid P \in \cF^{fcr} \rangle.$$ By \cite[Proposition 2.3]{BLO3}, $\Out_{\cF_n}(P_n) = \Out_\cF(P_n)$ is finite. Moreover  $T_n$ is the unique abelian maximal subgroup of $S_n$ (since $\cF$ is connected). Now by \cite[Theorem 2.8]{COS17}, $\cF_n$ is a saturated fusion system on $S_n$ uniquely determined by $\Aut_{\cF_n}(T_n)$, $T_n$, $\Aut_{S_n}(T_n)$ and $\cF_n^{cr} \backslash \{S_n,T_n\} \subseteq \cH_n \cup \cB_n$ where $$\cH_n=\{Z(S_n)\langle x \rangle \mid x \in S_n \backslash T_n\} \mbox{ and } \cB_n=\{Z_2(S_n)\langle x \rangle \mid x \in S_n \backslash T_n\}.$$ Moreover $\cF_n \subseteq \cF_{n+1}$ for all $n \ge 1$ by construction. The result follows.
\end{proof}

\section{Proofs of Theorems \ref{t:main} and \ref{t:main2}}\label{s:main2}

Combining the results in Sections \ref{s:prelim} and \ref{s:main} we can now prove:

\begin{thm}
Let $\cF$ be a simple $\ell$-local compact group on $S$ with discrete torus $T$ of index $\ell$ and Weyl group $W=\Aut_\cF(T)$. Then $\cF$ is connected and $$\bw(\cF)=|\Irr(W)|.$$
\end{thm}

\begin{proof}
By the remarks following Theorem \ref{t:or2} we may assume that $T$ is $\cF$-centric radical. By Theorem \ref{t:or1}, $\cF$ is connected and either $\cF^{cr}=\{S,T\} \cup \cH$ or $\cF^{cr}=\{S,T\} \cup \cB$. Set $U=\Aut_S(T)$ for short. By Theorem \ref{t:or1} and Lemma \ref{l:thev}, for any representative $Q \in \cH$ or $Q \in \cB$ we have, $$\begin{array}{rcl} |\Irr(W)|-\bw(\cF) &=& |\Irr(W)|-(z(kW)+z(k\Out_\cF(Q))+z(k\Out_\cF(S)) \medskip\\ &=& |\Irr(N_W(U))|- |\Irr(\Out_\cF(S))|-z(k\Out_\cF(Q)). 
\end{array}$$

The descriptions of the groups $N_W(U), \Out_\cF(S)$ and $\Out_\cF(Q)$ in Theorem \ref{t:or1} combined with Lemma \ref{l:chars} then yield:

$$\begin{array}{rcl}
|\Irr(W)|-\bw(\cF) &=& \ell|\Irr(X)|-(\ell-1)|\Irr(X)|-z(k(\SL_2(\ell) X)) \medskip\\
&=& |\Irr(X)|-z(k(\SL_2(\ell) X)) \medskip\\ &=& 0,
 \end{array}$$ where $X$ is as in Lemma \ref{l:outfp}. The result follows from this.
\end{proof}

\begin{thm}
Let $\cF$ be a simple $\ell$-local compact group with $v_\ell(|W|)=1$ and write $\displaystyle\cF = \lim_{\substack{\longrightarrow \\ n \ge 1}} \cF_n$ for saturated fusion systems $\cF_n$ on $S_n \le S$ as described in Proposition \ref{p:approx}. For each $n$ we have, $$|\cM(\cF_n)| = |\cP(\cF_n)|.$$
\end{thm}

\begin{proof}
 Let $n \in \NN$ be fixed and set $T_n:=T \cap S_n$, $W=\Aut_\cF(T_n)=\Aut_{\cF_n}(T_n)$ and $U=\Aut_{S_n}(T_n) \le W$.  We have two bijections $$\cM(\cF_n) \longleftrightarrow \Irr(S_n^{ab}:\Out_\cF(S_n)) \longleftrightarrow \Irr(T_nN_W(U)),$$ the first from the discussion before Lemma \ref{l:goltmot} and the second from the isomorphism $S_n^{ab}:\Out_\cF(S_n) \cong T_nN_W(U)$ (implied by Theorems \ref{t:or2} and \ref{t:or1}). By Lemma \ref{l:molg} this latter set is in bijection with equivalence classes of pairs  $$\{(\psi,\chi) \mid \psi \in \Irr(T_n), \chi \in \Irr(C_{N_W(U)}(\psi)) \mid [S_n,S_n] \le \ker(\psi)\}/ \sim_{\cF_n}$$  where $(\psi, \chi) \sim_{\cF_n} (\psi',\chi')$ if there is $\rho \in N_W(U)$ with $\psi^\rho=\psi'$ and $\chi^{\rho^{-1}}=\chi'$.  By the Glauberman-Isaacs correspondence there exists a $N_W(U)$-equivariant bijection between $T_n$ and $\Irr(T_n)$ which maps $Z(S_n)=C_{T_n}(U)$ to $\{\psi \in  \Irr(T_n) \mid [S_n,S_n] \le \ker(\psi)\}$, and hence $\cM(\cF_n)$ is in bijection with
$$\{(s, \chi) \mid s \in Z(S_n), \chi \in \Irr(C_{N_W(U)}(s)) \} /\sim_{\cF_n}$$ where $(s,\chi) \sim_{\cF_n} (s',\chi')$ if and only if there is $\phi \in N_W(U)$ with $s'=\phi(s)$ and $\chi'=\chi^{\phi^{-1}}$. Now since every morphism between elements of $Z(S_n)$ extends to an $\cF_n$-automorphism of $S_n$ and $|\Irr(C_{N_W(U)}(s))|=|\Irr_0(C_{N_W(U)}(s))|=|\Irr_0(C_W(s))|$ for each $s \in Z(S_n)$ (by Lemma \ref{l:thev}), we have $|\cM(\cF_n)|=$ $$\sum_{s \in Z(S_n)/N_W(U)} |\Irr(C_{N_W(U)}(s))|  = \sum_{s \in Z(S_n)/\cF_n} |\Irr_0(C_W(s))| = |\cP(\cF_n)|, $$ as required.

\end{proof}

\end{document}